\documentclass[11pt]{amsart}
\usepackage[margin=1.35in]{geometry}
\usepackage{amscd,amsmath,amsxtra,amsthm,amssymb,stmaryrd,xr,mathrsfs,mathtools,enumerate,commath, comment}
\usepackage{stmaryrd}
\usepackage{xcolor}
\usepackage{commath}
\usepackage{comment}
\usepackage{tikz-cd}
\usepackage{longtable} 
\usepackage{pdflscape} 
\usepackage{booktabs}
\usepackage{hyperref}
\definecolor{vegasgold}{rgb}{0.77, 0.7, 0.35}
\definecolor{darkgoldenrod}{rgb}{0.72, 0.53, 0.04}
\definecolor{gold(metallic)}{rgb}{0.83, 0.69, 0.22}
\hypersetup{
 colorlinks=true,
 linkcolor=darkgoldenrod,
 filecolor=brown,      
 urlcolor=gold(metallic),
 citecolor=darkgoldenrod,
 pdftitle={Asymptotic growth of Iwasawa Invariants in Noncommutative Towers},
 }
\usepackage[all,cmtip]{xy}

\DeclareFontFamily{U}{wncy}{}
\DeclareFontShape{U}{wncy}{m}{n}{<->wncyr10}{}
\DeclareSymbolFont{mcy}{U}{wncy}{m}{n}
\usepackage[T2A,T1]{fontenc}
\usepackage[OT2,T1]{fontenc}

\newcommand{\overbar}[1]{\mkern 1.5mu\overline{\mkern-1.5mu#1\mkern-1.5mu}\mkern 1.5mu}

\theoremstyle{oupplain}
\newtheorem{theorem}{Theorem}[section]
\newtheorem{lemma}[theorem]{Lemma}
\theoremstyle{oupdefinition}

\theoremstyle{oupremark}

\theoremstyle{oupproof}

\usepackage{tikz-cd}

\usepackage{enumitem}

\numberwithin{equation}{section}
 
\usepackage{color}

\newcommand{\Z}{\mathbb{Z}}

\newcommand{\Q}{\mathbb{Q}}

\newcommand{\op}[1]{\operatorname{#1}}

\theoremstyle{plain}
 \theoremstyle{definition}
 \newtheorem{Th*}{Theorem}
\newtheorem{Th}{Theorem}[section]
\newtheorem{Lemma}[Th]{Lemma}

\newtheorem{Proposition}[Th]{Proposition}
\newtheorem{Remark}[Th]{Remark}
 \theoremstyle{definition}
\newtheorem{Definition}[Th]{Definition}

\begin{document}

\title[Asymptotic growth of Iwasawa invariants in Noncommutative towers]{Asymptotic growth of Iwasawa invariants in Noncommutative towers of number fields}

\author[A.~Ray]{Anwesh Ray}
\address[Ray]{Department of Mathematics\\
University of British Columbia\\
Vancouver BC, Canada V6T 1Z2}
\email{anweshray@math.ubc.ca}

\maketitle

\begin{abstract}

Let $p$ be an odd prime, $F$ be a number field and consider a uniform infinite pro-$p$ extension $F_\infty$ of $F$ with Galois group $G=\op{Gal}(F_\infty/F)$. Let 
\[G=G_0\supset G_1\supset\dots \supset G_n\supset G_{n+1}\supset \dots\] be the descending $p$ central series of $G$ and set $F_n:=F_\infty^{G_n}$. Assume that $G$ is uniform and that $F_\infty$ contains the cyclotomic $\Z_p$-extension of $F$. Denote by $A_n$ the $p$-primary part of the class group of the cyclotomic $\Z_p$-extension of $F_n$. The $\lambda$-invariant of $F_n$ coincides with the corank of $A_n$ as a $\Z_p$-module. Assume that the Iwasawa $\mu$-invariant of the cyclotomic $\Z_p$-extension of $F$ equal to $0$. Then, the $\mu$-invariant of the cyclotomic $\Z_p$-extension of $F_n$ is $0$ as well and $A_n$ is isomorphic to $\left(\Q_p/\Z_p\right)^{\lambda_n}$. We study the asymptotic growth of $\lambda_n$ as $n$ goes to $\infty$.\end{abstract}

\section{Introduction}
\par Let $p$ be an odd prime number and let $F$ be a number field. The cyclotomic $\Z_p$-extension $F^{\op{cyc}}$ of $F$ is the unique $\Z_p$-extension of $F$ contained in the infinite cyclotomic field $F\left(\mu_{p^\infty}\right)$. Given $n\geq 0$, let $F_n$ be the subfield of $F^{\op{cyc}}$ such that $[F_n:F]=p^n$. Let $p^{e_n}$ be the exact power of $p$ dividing the class number of $F_n$. Iwasawa showed that for $n\gg 0$, 
\[e_n=\mu p^n+\lambda n+\nu,\] where $\mu\geq 0$, $\lambda\geq 0$ and $\nu$ are invariants that are independent of $n$. We denote the $\mu$ and $\lambda$-invariants by $\mu_p(F)$ and $\lambda_p(F)$ respectively. It is conjectured that $\mu_p(F)$ is equal to $0$ for all number fields $F$. This conjecture has been proven for abelian extensions $F$ by Ferrero and Washington \cite{ferrero1979iwasawa}.
\par We study asymptotic questions for class groups in certain infinite towers of number field extensions. Let $F_\infty$ be an infinite uniform pro-$p$ extension of with Galois group $G=\op{Gal}(F_\infty/F)$. Assume that $F_\infty$ contains $F^{\op{cyc}}$. Let 
\[G=G_0\supset G_1\supset\dots \supset G_n\supset G_{n+1}\supset \dots\] be the descending $p$ central series of $G$ and set $F_n:=F_\infty^{G_n}$ to be the fixed field of $G_n$. Denote by $A_n$ the $p$-primary part of the class group of $F_n^{\op{cyc}}$. It follows from a result of Iwasawa that there is an isomorphism of $\Z_p$-modules
\[A_n=\left(\Q_p/\Z_p\right)^{\lambda_n}\oplus A',\]where $A'$ has bounded exponent, i.e., $p^N A'=0$ for some $N>0$. Furthermore, $A'=0$ if and only if $\mu_p(F_n)=0$, and $\lambda_n=\lambda_p(F_n)$ (see Theorem \ref{iwasawa thm minor}). We state the main result of the paper.

\begin{Th*}[Theorem \ref{main thm}]\label{main thm intro}
Let $F$ be a number field and $p$ be an odd prime number. Let $F_\infty$ be a uniform pro-$p$ extension of $F$ such that 
\begin{itemize}
    \item $F_\infty$ contains $F^{\op{cyc}}$,
    \item all but finitely many primes of $F$ are ramified in $F_\infty$.
\end{itemize}
Let $d$ be the \emph{dimension} of $G$, and $S(F^{\op{cyc}})$ be the set of primes of $F^{\op{cyc}}$ that ramify in $F_\infty$. Then, $S(F^{\op{cyc}})$ is finite and the following assertions hold:
\begin{enumerate}
    \item $\mu_p(F_n)=0$ for all $n\geq 0$, 
    \item $A_n=\left(\Q_p/\Z_p\right)^{\lambda_p(F_n)}$ for all $n\geq 0$,
    \item for $n\geq 0$, $\lambda(F_n)$ satisfies the bounds \begin{equation}\begin{split}& \lambda_p(F_n)\geq p^{n(d-1)}\left(\lambda_p(F)+\xi_n^-\right)-\xi_n^-,\\
    & \lambda_p(F_n)\leq p^{n(d-1)}\left(\lambda_p(F)+\#S(F^{\op{cyc}})+\xi_n^+\right)-\xi_n^+,\end{split}\end{equation}
\end{enumerate}
where $\xi_n^+$ and $\xi_n^-$ are the quantities specified by \eqref{zetan plus minus definition}.
\end{Th*}

The quantities $\xi_n^+$ and $\xi_n^-$ arise from certain obstructions to the capitulation of ideals and ideal classes for the successive extension $F_{n+1}^{\op{cyc}}/F_n^{\op{cyc}}$. These constants are natural to define, however, difficult to compute. We expect that the quantities $\xi_n^+$ and $\xi_n^-$ are bounded as $n$ goes to $\infty$, and refer to the discussion in Remark \ref{last remark} for further details.
\par We note here that similar questions have been studied for $\Z_p^d$-extensions of a number field by Cuoco and Monsky (see \cite{cuoco1980growth, cuoco1981class}). The results of this article apply to a more general class of pro-$p$ extensions. We refer to the example in section \ref{s 4.2} of a noncommutative tower of number field extensions for which the assumptions of Theorem \ref{main thm intro} hold.

\par \emph{Organization:} Including the introduction, the article consists of four sections. In section \ref{s 2}, we introduce preliminary notions. In section \ref{s 3}, we introduce cohomological obstructions to capitulation and an Iwasawa's analog of the Riemann-Hurwitz formula for $\lambda$-invariants. In section \ref{s 4}, we prove the main result and illustrate it through an example. We also discuss a family of examples which arise from $p$-primary torsion fields generated by elliptic curves.

\section{Preliminaries}\label{s 2}
\par This section is dedicated to the discussion of preliminary notions. In $\S$\ref{s 2.1}, we introduce pro-$p$ towers of number field extensions for which we prove our results. In $\S$\ref{s 2.2} we introduce the Iwasawa $\mu$ and $\lambda$-invariants for the cyclotomic $\Z_p$-extension of a number field.
\subsection{Uniform pro-$p$ towers of number fields}\label{s 2.1}
\par Let $p$ be an odd prime number and let $F$ be a number field. We fix an algebraic closure $\overbar{F}$ of $F$. Let $F_\infty\subset \overbar{F}$ be an infinite pro-$p$ Galois extension of $F$ with Galois group $G:=\op{Gal}(F_\infty/F)$. The lower central $p$-series $\{G_n\}$ is defined as follows
\[G_0:=G\text{ and }G_{n+1}:=G_n^p[G_n, G].\]We shall also assume that $G$ is \emph{uniform}, i.e., the following conditions are satisfied
\begin{itemize}
    \item $G$ is finitely generated, 
    \item $G$ is \emph{powerful}, i.e., $[G,G]\subseteq G^p$, 
    \item for all $n\geq 0$, we have that \[[G_n:G_{n+1}]=[G:G_{1}].\]
\end{itemize}
Moreover, we assume that
\begin{description}
    \item[{\parbox[t]{0.75cm}{(C1)}}]\label{C1} $F_\infty$ contains $F^{\op{cyc}}$,
    \item[{\parbox[t]{0.75cm}{(C2)}}]\label{C2} all but finitely many primes of $F$ are ramified in $F_\infty$.
\end{description} Let $S$ be the set of primes $v\nmid p$ of $F$ are are ramified in $F_\infty$. Note that by assumption, $S$ is finite. Let $S(F^{\op{cyc}})$ be the set of primes of $F^{\op{cyc}}$ that lie above the set $S$. Any prime of $F$ is finitely decomposed in $F^{\op{cyc}}$, hence, $S(F^{\op{cyc}})$ is finite as well. Setting $[G:G_{1}]=p^d$, we find that $[G_{n}:G]=p^{nd}$ for all $n\geq 0$. The quantity $d(G)=d$ is referred to as the \emph{dimension of $G$}. The filtration $\{G_n\}$ of $G$ gives rise to a filtration of $F_\infty/F$ by a tower of number fields
\[F=F_0\subset F_1\subset \dots \subset F_n \subset F_{n+1}\subset \dots, \] where $F_n:=F_\infty^{G_n}$. For $n\geq 0$, we have the isomorphisms \[\op{Gal}(F_{n+1}/F_n)\simeq G_n/G_{n+1}\simeq \left(\Z/p\Z\right)^d.\]
Consider the following filtration
\[F^{\op{cyc}}=F_0^{\op{cyc}}\subset \dots \subset F_n^{\op{cyc}}\subset F_{n+1}^{\op{cyc}}\subset \dots,  \] where we recall that $F_n^{\op{cyc}}$ is the cyclotomic $\Z_p$-extension of $F_n$. Since it is assumed that $F^{\op{cyc}}$ is contained in $F_\infty$, it follows that each of the extensions $F_n^{\op{cyc}}$ is contained in $F_\infty$.
\par Given $n\geq 0$, we set $H_n$ to be the Galois group $\op{Gal}\left(F_\infty/F_n^{\op{cyc}}\right)$. For $n\geq 0$, $H_n$ is a subgroup of $G_n$ and $H_{n+1}\subset H_{n}$. Set $H=H_0$ and $\Gamma:=\op{Gal}(F^{\op{cyc}}/F)$, and identify $\Gamma$ with $G/H$. Let $\{\Gamma_n\}$ be the descending central series of $\Gamma$. The normal subgroup $H$ is uniform and $d(H)=(d(G)-1)$. Moreover, the series $\{H_n\}$ coincides with the descending central series of $H_n$ and $\Gamma_n$ is identified with $G_n/H_n$ (see \cite[Theorem 3.6]{DS} and \cite[Lemma 2.6]{HL} for details). For $n\geq 0$, we therefore find that $\op{Gal}(F_{n+1}^{\op{cyc}}/F_{n}^{\op{cyc}})=H_n/H_{n+1}$ is isomorphic to $\left(\Z/p\Z\right)^{(d-1)}$ where $d=d(G)$.
\subsection{Iwasawa Invariants}\label{s 2.2}
\par Let $F_\infty/F$ be a uniform pro-$p$ extension of $F$ which satisfied the conditions (C1) and (C2) specified in $\S \ref{s 2.1}$. Given $n\geq 0$, we identify $\Gamma_n$ with $\op{Gal}\left(F_n^{\op{cyc}}/F_n\right)$ and choose a topological generator $\gamma_n$ of $\Gamma_n$. The Iwasawa algebra \[\Lambda(\Gamma_n):=\varprojlim_m [ \Gamma_n/\Gamma_n^{p^m}]\] is identified with the formal power ring $\Z_p\llbracket T\rrbracket$, where $T$ is a formal variable. This isomorphism is precribed by mapping the generator $(\gamma_n-1)$ to $T$.
\par Associated to a finitely generated and torsion $\Z_p\llbracket T\rrbracket $-module $M$ are its associated Iwasawa invariants $\mu(M)\geq 0$ and $\lambda(M)\geq 0$. In greater detail, let $f_M(T)\in \Z_p\llbracket T\rrbracket $ be the characteristic element associated to $M$, which we recall is well defined up to multiplication by a unit in $\Z_p\llbracket T\rrbracket$. According to the Weierstrass Preparation theorem, $f(T)$ may be factored as a product $f(T)=p^{\mu} P(T)u(T)$, where $P(T)$ is a distinguished polynomial and $u(T)$ is a unit in $\Z_p\llbracket T\rrbracket$. The $\mu$-invariant $\mu(M)$ is the quantity $\mu\geq 0$, in the above factorization, and $\lambda(M)$ is the degree of $P(T)$. Since $f(T)$ is uniquely determined up to a unit in $\Z_p\llbracket T\rrbracket$, and the above factorization is unique, it follows that the Iwasawa invariants $\mu(M)$ and $\lambda(M)$ are well defined. 
\par The modules we shall consider do arise from the $p$-primary parts of ideal classes in the cyclotomic tower of $F_n$. Given a number field $K$, let $\op{Cl}_p(K)$ be the $p$-Sylow subgroup of the class group of $K$. Denote by $X_n=X(F_n^{\op{cyc}})$ the inverse limit $\varprojlim_K \op{Cl}_p(K)$, where $K$ ranges over all number field extensions of $F_n$ that are contained in $F_n^{\op{cyc}}$. The inverse limit is taken with respect to norm maps. Denote by $\mu_p(K_n)$ and $\lambda_p(K_n)$ the associated $\mu$ and $\lambda$-invariants of $X_n$ as a $\Lambda(\Gamma_n)$-module. It is easy to see that $\mu_p(K_n)=0$ if and only if $X_n$ is finitely generated as a $\Z_p$-module. Moreover, it is easy to see that $\lambda_p(X_n)$ coincides with the rank of $X_n$ as a $\Z_p$-module (see \cite[Proposition 13.25]{washington1997}). 
\section{Iwasawa's Riemann-Hurwitz formula}\label{s 3}
\par Throughout, let $p$ will be an odd prime number. In this section, we introduce the methods used in proving our main results. The main contribution comes from an analog of the Riemann-Hurwitz formula due to Iwasawa \cite{iwasawa1981riemann}. The result generalizes Kida's formula \cite{kida1980extensions}, which is proven for CM-fields. In section \ref{s 3.1}, we shall introduce certain cohomology classes which realize the obstructions to capitulation of ideals and ideal classes. In section \ref{s 3.2}, we state the results of Iwasawa, which shall be used in proving our main result.

\subsection{Obstructions to Capitulation of ideals and ideal classes}\label{s 3.1}
\par Capitulation cokernels are of natural interest in Iwasawa theory, for instance, the reader may refer to \cite{le2005capitulation} where the relationship between capitulation and Greenberg's conjecture is studied. This section introduces cohomological obstructions to the capitulation of ideals and ideal classes. For further details, the reader is referred to \cite{iwasawa1981riemann, schettler2014generalizations}.
\par By a $\Z_p$-field, we mean a finite extension of $\Q^{\op{cyc}}$. Given a $\Z_p$-field $\mathcal{K}$ it is easy to see that there is a number field $K$ such that $\mathcal{K}=K^{\op{cyc}}$.
\begin{Lemma}\label{Iwasawa invariants are well defined}
Let $K$ and $K'$ be number fields such that $K\subset K'$ and $K^{\op{cyc}}=K'^{\op{cyc}}$. Then, $\mu_p(K)=0$ if and only if $\mu_p(K')=0$, and furthermore, $\lambda_p(K)=\lambda_p(K')$.
\end{Lemma}
\begin{proof}
Let $X(K^{\op{cyc}})$ be the inverse limit $\varprojlim_E \op{Cl}_p(E)$, where $E$ ranges through all number fields contained in $K^{\op{cyc}}$. The $\lambda$-invariant $\lambda_p(K)$ is equal to the rank of $X(K^{\op{cyc}})$ and hence, $\lambda_p(K)=\lambda_p(K')$. It is easy to see from the structure theory of Iwasawa modules that $\mu_p(K')=[K':K]\mu_p(K)$, hence, $\mu_p(K)=0$ if and only if $\mu_p(K')=0$.
\end{proof}Given a $\Z_p$-field $\mathcal{K}$, we define $\lambda_{\mathcal{K}}$ to be $\lambda_p(K)$ for any number field for whuch $\mathcal{K}=K^{\op{cyc}}$. Furthermore, we may write $\mu_{\mathcal{K}}=0$ to mean that $\mu_p(K)=0$ for any number field $K$ such that $\mathcal{K}=K^{\op{cyc}}$. Lemma \ref{Iwasawa invariants are well defined} shows that these notions are well defined.
\par Let $\mathcal{K}$ be a $\Z_p$-field and let $\mathcal{O}_{\mathcal{K}}$ be the ring of algebraic integers of $\mathcal{K}$. Set $I_{\mathcal{K}}$ to denote the group of ideals of $\mathcal{O}_{\mathcal{K}}$ and let $P_{\mathcal{K}}$ be the subgroup of principal ideals of $I_{\mathcal{K}}$. The ideal class group $\op{Cl}_{\mathcal{K}}$ is the quotient $I_{\mathcal{K}}/P_{\mathcal{K}}$. Let $\mathcal{L}$ and $\mathcal{K}$ be $\Z_p$-fields such that $\mathcal{K}$ is contained in $\mathcal{L}$. The natural embedding $I_{\mathcal{K}}\rightarrow I_{\mathcal{L}}$ induces a map $\op{Cl}_{\mathcal{K}}\rightarrow \op{Cl}_{\mathcal{L}}$. 

\begin{Lemma}\label{lemma 3.1}
Let $\mathcal{L}/\mathcal{K}$ be a Galois extension of $\Z_p$-fields and set $G=\op{Gal}(\mathcal{L}/\mathcal{K})$. The following assertions hold
\begin{enumerate}
    \item $H^1(\mathcal{L}/\mathcal{K}, \mathcal{O}_{\mathcal{L}}^\times)=P_{\mathcal{L}}^G/P_{\mathcal{K}}$, 
    \item $\op{Ker}\left(H^2(\mathcal{L}/\mathcal{K}, \mathcal{O}_{\mathcal{L}}^\times)\longrightarrow H^2(\mathcal{L}/\mathcal{K}, \mathcal{L}^\times)\right)\simeq \op{Coker}\left(I_{\mathcal{L}}^G\rightarrow C_{\mathcal{L}}^G\right)$.
\end{enumerate}
\end{Lemma}
\begin{proof}
For the proof of the above result, the reader is referred to \cite[p. 272]{iwasawa1981riemann}.
\end{proof}

\begin{Lemma}\label{lemma 3.2}
Let $\mathcal{K}$ be a $\Z_p$-field and let $\mathcal{L}$ be a finite $p$-extension of $\mathcal{K}$. Assume that $\mathcal{L}/\mathcal{K}$ is unramified at all archimedian places of $\mathcal{K}$. Then, $H^n(\mathcal{L}/\mathcal{K}, \mathcal{L}^\times)=0$ for all $n\geq 1$.
\end{Lemma}
\begin{proof}
The above result is \cite[Lemma 5]{iwasawa1981riemann}.
\end{proof}
Note that in our applications we assume that $p$ is odd, hence, any $p$-extension $\mathcal{L}/\mathcal{K}$ is unramified at all archimedian places of $\mathcal{K}$.
\begin{Proposition}\label{prop 3.3}
Let $p$ be an odd prime and let $\mathcal{L}/\mathcal{K}$ be a Galois $p$-extension of $\Z_p$-fields with $G=\op{Gal}(\mathcal{L}/\mathcal{K})$. Then, the following assertions hold
\begin{enumerate}
    \item $H^1(\mathcal{L}/\mathcal{K}, \mathcal{O}_{\mathcal{L}}^\times)=P_{\mathcal{L}}^G/P_{\mathcal{K}}$, 
    \item $H^2(\mathcal{L}/\mathcal{K}, \mathcal{O}_{\mathcal{L}}^\times)= \op{Cl}_{\mathcal{L}}^G/I_{\mathcal{L}}^G$. 
\end{enumerate}\end{Proposition}

\begin{proof}
The above result is a direct consequence of the Lemmas \ref{lemma 3.1} and \ref{lemma 3.2}.
\end{proof}

Let $\mathcal{L}/\mathcal{K}$ be a Galois extension of $\Z_p$-fields such that $G=\op{Gal}(\mathcal{L}/\mathcal{K})$is a $p$-group. We refer to $P_{\mathcal{L}}^G/P_{\mathcal{K}}$ and $\op{Cl}_{\mathcal{L}}^G/\op{Cl}_{\mathcal{K}}$ as \emph{capitulation quotients} for the extension $\mathcal{L}/\mathcal{K}$ since they measure the obstruction to an ideal or ideal class of $\mathcal{O}_\mathcal{L}$ that is $G$-invariant to arise from an ideal or ideal class of $\mathcal{O}_\mathcal{K}$. According to Proposition \ref{prop 3.3}, $H^2(\mathcal{L}/\mathcal{K}, \mathcal{O}_{\mathcal{L}}^\times)= \op{Cl}_{\mathcal{L}}^G/I_{\mathcal{L}}^G$, and therefore is a quotient of $\op{Cl}_{\mathcal{L}}^G/\op{Cl}_{\mathcal{K}}$. As a result, we refer to non-zero cohomology classes in $H^i(\mathcal{L}/\mathcal{K}, \mathcal{O}_{\mathcal{L}}^\times)$ for $i=1,2$, as obstructions to capitulation. When the cohomology group $H^i(\mathcal{L}/\mathcal{K}, \mathcal{O}_{\mathcal{L}}^\times)$ is finite, we set $\left|H^i(\mathcal{L}/\mathcal{K}, \mathcal{O}_{\mathcal{L}}^\times)\right|$ to denote its cardinality. Note that this quantity is a power of $p$. We set \[h^i(\mathcal{L}/\mathcal{K}):=\op{log}_p \left(|H^i(\mathcal{L}/\mathcal{K}, \mathcal{O}_{\mathcal{L}}^\times)|\right).\] 

\subsection{Growth of Iwasawa invariants}\label{s 3.2}

\par Given a $\Z_p$-field $\mathcal{F}$, recall that $\op{Cl}_{\mathcal{F}}$ denotes its class group. Set $A_{\mathcal{F}}$ denote the $p$-primary part of $\mathcal{F}$. Let $F$ be a number field and let $\mathcal{F}$ denote $F^{\op{cyc}}$. The Iwasawa $\mu$ and $\lambda$-invariants for the extension $\mathcal{F}/F$ are denoted $\mu_p(F)$ and $\lambda_p(F)$ respectively. The following result of Iwasawa gives an explicit description for $A_{\mathcal{F}}$ in terms of the Iwasawa invariants $\mu_p(F)$ and $\lambda_p(F)$.
\begin{Th}[Iwasawa]\label{iwasawa thm minor}
Let $\mathcal{F}=F^{\op{cyc}}$ be a $\Z_p$-field. Then, there is an isomorphism of $\Z_p$-modules
\[A_{\mathcal{F}}=\left(\Q_p/\Z_p\right)^{\lambda_{\mathcal{F}}}\oplus A',\]where $A'$ has bounded exponent, i.e., $p^N A'=0$ for some $N>0$. Furthermore, the following assertions hold
\begin{enumerate}
    \item $A'=0$ if and only if $\mu_p(F)=0$, 
    \item $\lambda_{\mathcal{F}}=\lambda_p(F)$.
\end{enumerate}
\end{Th}
\begin{proof}
The reader is referred to \cite[p.272]{iwasawa1981riemann} or \cite[Theorem 1.1]{schettler2014generalizations} for the statement of the above result. The result was originally proven in \cite{iwasawa1959gamma, iwasawa1973zl}.
\end{proof}

\begin{Th}[Iwasawa]\label{iwasawa thm major}
Let $p$ be an odd prime and $L/K$ a Galois extension of number fields such that $\op{Gal}(L/K)$ is isomorphic to $\Z/p\Z$. Assume that $\mu_p(K)=0$. Then, the following assertions hold:
\begin{enumerate}
    \item $\mu_p(L)=0$,
    \item $h^i:=h^i(L^{\op{cyc}}/K^{\op{cyc}})$ are finite for $i=1,2$,
    \item $\lambda_p(L)=p\lambda_p(K)+\sum_{w\nmid p} (e(w)-1)+(p-1)\left(h^2-h^1\right)$.
\end{enumerate}
Here $w$ runs over all primes of $L^{\op{cyc}}$ that do not lie above $p$ and $e(w)$ refers to the ramification index over $K^{\op{cyc}}$.
\end{Th}
\begin{proof}
The conclusion that $\mu_p(L)=0$ follows from \cite[Theorem 11.3.8]{neukirch2013cohomology}. The remaining assertions follow from the main result of \cite{iwasawa1981riemann}. The finiteness of $h^i$ for $i=1,2$ follows as a direct consequence of the arguments in \emph{loc. cit.}
\end{proof}

\section{Growth of Iwasawa invariants in pro-$p$ extensions}\label{s 4}
\par In this section, we prove the main result of this paper. The main ingredients are the Theorems \ref{iwasawa thm minor} and \ref{iwasawa thm major}. 

\subsection{Main results}
First, we introduce some further notation. Let $F$ be a number field and $F_\infty$ be a uniform pro-$p$ extension of $F$ satisfying the conditions (C1) and (C2) of section \ref{s 2.1}. Recall that the descending central series on $G$ induces the filtration $\{F_n\}$ on $F_\infty$, and that $\mu_p(F_n)$ and $\lambda_p(F_n)$ are the $\mu$ and $\lambda$-invariants of $F_n^{\op{cyc}}/F_n$.
\par We consider the tower of $\Z_p$-fields \[F^{\op{cyc}}=F_0^{\op{cyc}}\subset F_1^{\op{cyc}}\subset\dots F_n^{\op{cyc}}\subset F_{n+1}^{\op{cyc}}\subset \dots,\]
and note that for $n\geq 0$ \[\op{Gal}\left(F_{n+1}^{\op{cyc}}/F_n^{\op{cyc}}\right)=H_n/H_{n+1}\simeq \left(\Z/p\Z\right)^{(d-1)}.\]
For $n\geq 0$, let $\mathcal{F}_n$ to be the $\Z_p$-field $F_n^{\op{cyc}}$. Choose a filtration $\{\mathcal{F}_n^{(j)}\mid j=0,\dots, (d-1)\}$ of $\mathcal{F}_{n+1}/\mathcal{F}_n$ for which
\begin{itemize}
    \item $\mathcal{F}_n^{(j)}\subset \mathcal{F}_n^{(j+1)}$ is a Galois extension with $[\mathcal{F}_n^{(j+1)}:\mathcal{F}_n^{(j)}]=p$,
    \item $\mathcal{F}_n^{(0)}=\mathcal{F}_n$ and $\mathcal{F}_n^{(d-1)}=\mathcal{F}_{n+1}$.
\end{itemize}
Note that it follows from Theorem \ref{iwasawa thm major} that $h^i(\mathcal{F}_n^{(j+1)}/\mathcal{F}_n^{(j)})$ is well defined for $i=1,2$. We set \[B_n:=\sum_{j=0}^{d-2}p^{d-2-j}\left(h^2(\mathcal{F}_n^{(j+1)}/\mathcal{F}_n^{(j)})-h^1(\mathcal{F}_n^{(j+1)}/\mathcal{F}_n^{(j)})\right),\]
and set 
\begin{equation}\label{def of C_n}
    C_n:=(p-1)\sum_{i=0}^{n-1} p^{(d-1)(n-1-i)} B_i,
\end{equation}
and it is understood that $C_0=0$.

\begin{lemma}\label{lemma bounds on C_n}
With respect to notation above, letting
\begin{equation}\label{zetan plus minus definition}\begin{split}&\xi_n^+:= \op{max}\left\{h^2(\mathcal{F}_i^{(j+1)}/\mathcal{F}_i^{(j)})-h^1(\mathcal{F}_n^{(j+1)}/\mathcal{F}_n^{(j)})\mid i\leq n\text{ and }j\leq d-2\right\},\\
&\xi_n^-:= \op{min}\left\{h^2(\mathcal{F}_i^{(j+1)}/\mathcal{F}_i^{(j)})-h^1(\mathcal{F}_n^{(j+1)}/\mathcal{F}_n^{(j)})\mid i\leq n\text{ and }j\leq d-2\right\},\\\end{split}\end{equation}
we find that for $n\geq 0$, $C_n$ satisfies the bounds
\[\left(p^{n(d-1)}-1\right)\xi_n^-\leq C_n\leq \left(p^{n(d-1)}-1\right)\xi_n^+.\]
\end{lemma}
\begin{proof}
We prove that $C_n\leq \left(p^{n(d-1)}-1\right)\xi_n^+$. The proof of the lower bound is similar. For $i\leq n$, we find that
\[B_i\leq \sum_{j=0}^{d-2}p^{d-2-j}\xi_i^+=\frac{p^{d-1}-1}{p-1} \xi_i^+\leq \frac{p^{d-1}-1}{p-1} \xi_n^+.\]
Plugging this inequality into \eqref{def of C_n}, we obtain the bound
\[\begin{split}C_n &= (p-1)\sum_{i=0}^{n-1} p^{(d-1)(n-1-i)} B_i\\
&\leq \xi_n^+\sum_{i=0}^{n-1} \left(p^{(d-1)(n-i)}-p^{(d-1)(n-1-i)}\right) \\
&= \xi_n^+ \left(p^{n(d-1)}-1\right),\\
\end{split}\]which completes the proof.
\end{proof}
\par Consider the group $A_n:=A\left(F_n^{\op{cyc}}\right)$. According to Theorem \ref{iwasawa thm minor}, $A_n=\left(\Q_p/\Z_p\right)^{\lambda_n}\oplus A_n'$, where $\lambda_n=\lambda(F_n)$ and $A_n'$ has bounded exponent. Moreover, $A_n'=0$ if and only if $\mu_p(F_n)=0$. 

\begin{Th}\label{main thm}
Let $F$ be a number field and $p$ be an odd prime number. Let $F_\infty$ be a uniform pro-$p$ extension of $F$ satisfying conditions (C1) and (C2) of section \ref{s 2.1}. The following assertions hold:
\begin{enumerate}
    \item $\mu_p(F_n)=0$ for all $n\geq 0$, 
    \item $A_n=\left(\Q_p/\Z_p\right)^{\lambda_p(F_n)}$ for all $n\geq 0$,
    \item for $n\geq 0$, $\lambda(F_n)$ satisfies the bounds \begin{equation}\label{main eqns}\begin{split}& \lambda_p(F_n)\geq p^{n(d-1)}\left(\lambda_p(F)+\xi_n^-\right)-\xi_n^-,\\
    & \lambda_p(F_n)\leq p^{n(d-1)}\left(\lambda_p(F)+\#S(F^{\op{cyc}})+\xi_n^+\right)-\xi_n^+,\end{split}\end{equation}
\end{enumerate}
where $\xi_n^+$ and $\xi_n^-$ are the quantities specified by \eqref{zetan plus minus definition}.
\end{Th}

\begin{proof}
\par It follows from Theorem \ref{iwasawa thm major} and by induction on $n$ that $\mu(F_n)=0$. It is thus a consequence of Theorem \ref{iwasawa thm minor} that $A_n=\left(\Q_p/\Z_p\right)^{\lambda_p(F_n)}$ for all $n\geq 0$. Recall that for $n\geq 0$, the quantity $C_n$ is defined by \eqref{def of C_n}. According to Lemma \ref{lemma bounds on C_n}, we find that for $n\geq 0$,
\[\left(p^{n(d-1)}-1\right)\xi_n^-\leq C_n\leq \left(p^{n(d-1)}-1\right)\xi_n^+.\]Therefore in order to prove the bounds on $\lambda_p(F_n)$, it suffices to prove that
\begin{equation}\label{bounds for lambda in proof}p^{n(d-1)}\lambda_p(F)+C_n\leq \lambda_p(F_n)\leq p^{n(d-1)}\left(\lambda_p(F)+\#S(F^{\op{cyc}})\right)+C_n.\end{equation} Recall that for every value of $k$, we have chosen a filtration 
\[F_k^{\op{cyc}}=\mathcal{F}_k^{(0)}\subset \dots \subset \mathcal{F}_k^{(j)}\subset \mathcal{F}_k^{(j+1)}\subset\dots\subset \mathcal{F}_k^{(d-1)}=F_{n+1}^{\op{cyc}},\]
such that the successive extensions are Galois with Galois group isomorphic to $\Z/p\Z$. Concatenate the above filtrations to obtain a filtration
\[F^{\op{cyc}}=\mathcal{E}_0\subset \mathcal{E}_1\subset \dots \subset \mathcal{E}_i\subset \mathcal{E}_{i+1}\subset \dots \subset \mathcal{E}_{n(d-1)}=F_n^{\op{cyc}}.\]
Let $\lambda_i$ denote the $\lambda$-invariant of $\mathcal{E}_i$,
where $\mathcal{E}_{k(d-1)+j}:=\mathcal{F}_k^{(j)}$ for $0\leq j\leq d-1$. By induction, the $\mu$-invariant of each of the $\Z_p$-fields above is $0$. The relation 
\[\lambda_{i+1}=p\lambda_i+\sum_{w\nmid p} (e(w)-1)+(p-1)\left(h^2(\mathcal{E}_{i+1}/\mathcal{E}_i)-h^1(\mathcal{E}_{i+1}/\mathcal{E}_i)\right)\] follows from Theorem \ref{iwasawa thm major}. Recursively applying the above, the obtain the bounds \eqref{bounds for lambda in proof}.
\end{proof}

\begin{Definition}
We say that the successive obstructions to capitulation are uniformly bounded if there exists $B>0$ independent of $n$ such that for all $i\geq 0$\[\left|h^2(\mathcal{F}_{i+1}/\mathcal{F}_i)-h^1(\mathcal{F}_{i+1}/\mathcal{F}_i)\right|\leq B.\]
We say that the successive obstructions to capitulation eventually vanish if there exists $n_0$ such that for all $i\geq n_0$, 
\[h^2(\mathcal{F}_{i+1}/\mathcal{F}_i)=h^1(\mathcal{F}_{i+1}/\mathcal{F}_i)=0.\]
\end{Definition}
\begin{Remark}\label{last remark}
The above result shows that if there is a uniform bound $B>0$ independent of $n$ such that for all $i\geq 0$\[\left|h^2(\mathcal{F}_{i+1}/\mathcal{F}_i)-h^1(\mathcal{F}_{i+1}/\mathcal{F}_i)\right|\leq B,\]then, it follows that \[p^{n(d-1)}(\lambda_p(F)-B)<\lambda_p(F_n)< p^{n(d-1)}(\lambda_p(F)+\# S(F^{\op{cyc}})+B).\] At this point in time, we are not able to show the uniform boundedness of successive obstruction quotients in our examples since it is a difficult condition to computationally verify given that it concerns the structure of class groups over infinite Galois extensions. However, it is a question worth exploring further. We expect the formulas for the growth of the the $\Z_p$ corank of $A(F_n)$ should be asymptotic as a function of $n$ to $\lambda(G)p^{(d-1)n}$, where $\lambda(G)$ is a non-negative constant. More precisely, we expect that as a function of $n$, the quotient $A(F_n)/p^{(d-1)n}$ converges to a constant $\lambda(G)$. The quantity $\lambda(G)$ should be thought of as noncommutative analog of the classical $\lambda$-invariant. In light of the bounds for $\lambda_p(F_n)$ of Theorem \ref{main thm}, it is reasonable to expect that the successive obstruction to capitulation $\xi_n^+$ and $\xi_n^-$ introduced in this article, are bounded independent of $n$.
\end{Remark}
\subsection{An explicit example}\label{s 4.2}
\par We illustrate Theorem \ref{main thm} through an example. Let $p$ be an odd prime and let $\mu_p$ denote the $p$-th roots of unity and $F:=\Q(\mu_p)$. Let $\ell\neq p$ be a prime and set $\mu_{p^\infty}:=\cup_{n\geq 1} \mu_{p^n}$ and $\ell^{p^{-\infty}}:=\cup_{n\geq 1} \ell^{p^{-n}}$. Let $F_\infty$ be the \emph{false Tate extension} $\Q\left(\mu_{p^\infty}, \ell^{p^{-\infty}}\right)$. Observe that $F_\infty$ is a pro-$p$ extension of $F$ with $F_n=\Q\left(\mu_{p^{n+1}}, \ell^{p^{-n}}\right)$. The set $S$ is the set of primes $v$ of $F$ that lie above $\ell$. Note that since $F_\infty$ contains $F^{\op{cyc}}=\Q(\mu_{p^\infty})$, the condition (C1) of \eqref{s 2.1} is satisfied. The only primes that ramify in $F_\infty$ are the primes of $F$ above $\{p, \ell\}$. The condition (C2) is also satisfied. The dimension $d$ of $G$ is equal to $2$. 
\par Since $F$ is an abelian extension of $\Q$, it follows from the result of Ferrero and Washington \cite{ferrero1979iwasawa} that $\mu_p(F)=0$. According to Theorem \ref{main thm}, we find that
\begin{enumerate}
    \item $\mu_p(F_n)=0$ for all $n$, 
    \item the $p$-primary class group $A(F_n^{\op{cyc}})$ is isomorphic to $\left(\Q_p/\Z_p\right)^{\lambda_p(F_n)}$,
    \item for $n\geq 0$, $\lambda(F_n)$ satisfies the bounds \begin{equation}\begin{split}& \lambda_p(F_n)\geq p^{n}\left(\lambda_p(F)+\xi_n^-\right)-\xi_n^-,\\
    & \lambda_p(F_n)\leq p^{n}\left(\lambda_p(F)+\#S(F^{\op{cyc}})+\xi_n^+\right)-\xi_n^+,\end{split}\end{equation}
\end{enumerate}
Specialize the above example to $p=3$ and $\ell=5$. Since $5^2\equiv 1\mod{3}$ and $5^2\not \equiv 1\mod{3^2}$, we find that $\ell$ splits into $2$ primes $\mathfrak{l}$ and $\mathfrak{l}^*$ in $F$, and the primes $\mathfrak{l}$ and $\mathfrak{l}^*$ are both inert in $F^{\op{cyc}}$. Therefore, $\#S(F^{\op{cyc}})=2$, and the upper bound specializes to 
\[\lambda_3(F_n)\leq 3^{n}\left(\lambda_3(F)+2+\xi_n^+\right)-\xi_n^+.\]
\subsection{Torsion fields generated by Elliptic curves}
\par A large class of examples arise from infinite torsion fields generated by elliptic curves. Let $E$ be an elliptic curve defined over $\Q$ without complex multiplication and let $p$ be an odd prime. Let $E[p^n]$ be the group of $p^n$ torsion points of $E$ and consider the action of the absolute Galois group $\op{G}_{\Q}:=\op{Gal}(\bar{\Q}/\Q)$ on $E[p^n]$. Since $E[p^n]$ is isomorphic to $\left(\Z/p^n\Z\right)^2$, we obtain a Galois representation 
\[\varrho_{E, p^n}:\op{G}_{\Q}\rightarrow \op{GL}_2(\Z/p^n\Z).\] The 
$p$-adic Tate module $T_p(E)$ is the inverse limit $\varprojlim_n E[p^n]$ with respect to multiplication by $p$ maps $E[p^{n+1}]\rightarrow E[p^n]$. The Galois representation on $T_p(E)$ is the inverse limit of the Galois representations $\varrho_{E, p^n}$, and is expressed as follows
\[\rho_{E,p}:\op{G}_{\Q}\rightarrow \op{GL}_2(\Z_p).\] Let $F_\infty$ be the fixed field of the kernel of $\rho_{E,p}$. We identify $\op{Gal}(F_\infty/\Q)$ with the image of $\rho_{E,p}$ in $\op{GL}_2(\Z_p)$. Let $F_n$ be the number field which is fixed by the kernel of $\varrho_{E,p^{n+1}}$ and set $F:=F_0$. Set $\Gamma(p)$ to be the kernel of the reduction modulo-$p$ map
\[\op{GL}_2(\Z_p)
\rightarrow \op{GL}_2(\Z/p\Z).\] Then, $G:=\op{Gal}(F_{\infty}/F)$ may be identified with the intersection of $\op{im}\rho_{E,p}$ with $\Gamma(p)$, and thus, is a pro-$p$ group.
\par It follows from Serre's Open image theorem (see \cite{serre1972proprietes}) that the image of $\rho_{E,p}$ must contain $\Gamma(p)$ for all but finitely many primes $p$. Suppose that $p$ is a prime such that $\rho_{E,p}$ contains $\Gamma(p)$. Then, $G$ is isomorphic to $\Gamma(p)$, and it is easy to see that $G_n$ is isomorphic to 
\[\Gamma(p^{n+1}):=\op{Ker}\left(\op{GL}_2(\Z_p)
\rightarrow \op{GL}_2(\Z/p^{n+1}\Z)\right).\] Note that since $G$ is a finite index subgroup of $\op{GL}_2(\Z_p)$, its dimension $d$ is equal to $4$. Furthermore, the set of primes $S$ is the set of primes $v$ of $F$, such that $v|\ell$ for a prime $\ell\neq p$ at which $E$ has bad reduction. Clearly this set $S$ is finite. Furthermore, since the determinant of $\rho_{E,p}$ is the cyclotomic character, it follows that $F_\infty$ contains $F^{\op{cyc}}=\Q(\mu_{p^\infty})$. Therefore, the assumptions of Theorem \ref{main thm} are satisfied provided $\mu_p(F)=0$. This is expected to be true. Assuming that $\mu_p(F)=0$, the following bounds for $\lambda_p(F_n)$ hold
\begin{equation}\begin{split}& \lambda_p(F_n)\geq p^{3n}\left(\lambda_p(F)+\xi_n^-\right)-\xi_n^-,\\
    & \lambda_p(F_n)\leq p^{3n}\left(\lambda_p(F)+\#S(F^{\op{cyc}})+\xi_n^+\right)-\xi_n^+,\end{split}\end{equation}
\bibliographystyle{abbrv}
\bibliography{references}

\end{document}